\documentclass[a4paper,11pt]{article}

\usepackage[OT1]{fontenc}

\usepackage{graphicx}

\usepackage{amssymb}
\usepackage[latin1]{inputenc}

 \newtheorem{definition}{Definition}[section]

\newtheorem{remark}{Remark}[section]
\newtheorem{lemma}{Lemma}[section]
\newtheorem{corollary}{Corollary}[section]
\newtheorem{proposition}{Proposition}[section]
\newtheorem{theorem}{Theorem}[section]

\newenvironment{proof}{\noindent\textbf{Proof: }}{\hfill \small $\Box$}

\newcommand{\NN}{\mathbb{N}}
\newcommand{\RR}{\mathbb{R}}
\newcommand{\II}{\mathbb{I}}
\newcommand{\lra}{\longrightarrow}
\newcommand{\Ra}{\Rightarrow}
\newcommand{\prt}[1]{\langle #1\rangle}
\newcommand{\brs}[1]{\mid\! #1\!\mid}
\def\bigsqcap{\mathop{\rule[-0.2ex]{.07em}{2.17ex}
               \rule[1.9ex]{0.55em}{.17ex}
                \rule[-0.2ex]{.07em}{2.17ex}}}
\def\Bigsqcap{ \mathop{ \rule[-0.7ex]{.07em}{3ex}
                \rule[2.15ex]{0.8em}{.17ex}
                \rule[-0.7ex]{.07em}{3ex}}}

\newcommand{\Sumat}[2]{\sum\limits_{#1}^{#2}}

\newcommand{\Lim}[1]{\lim\limits_{#1}}

\usepackage{color}

\newcommand{\rojo}[1]{\textcolor[rgb]{1.00,0.00,0.00}{#1}}

\begin{document}


\title{Interval probability density functions constructed from a generalization of the Moore and Yang integral}

\author{Benjamin Bedregal \\ Departamento de Inform\'atica e Matem\'atica Aplicada  \\
	Universidade Federal do Rio Grande do Norte \\
	Natal, RN, Brazil, bedregal@dimap.ufrn.br \\
[0.4cm] Claudilene Gomes da Costa \\ Centro de Ci\^encias Exatas \\
Universidade Federal da Para\'iba  \\
Rio Tinto, PB, Brazil, claudilene@dce.ufpb.br \\ [0.4cm]  Eduardo Palmeira \\ Departamento de Ci\^encias Exatas e Tecnol\'ogicas \\
 Universidade Estadual de Santa Cruz, \\  Ilh\'eus, BA, Brazil, espalmeira@uesc.br \\ [0.4cm] Edmundo Mansilla \\
 Departamento de Matem\'atica y F\'isica \\ Universidad de Magallanes \\ Punta Arenas, Chile, edmundo.mansilla@umag.cl}

%
%
%
%
%
%

\maketitle

\begin{abstract}
Moore and Yang defined an integral notion, based on an extension of Riemann sums,  for inclusion monotonic continuous interval functions, where the integrations limits are real numbers. This integral notion extend the usual integration of real functions based on Riemann sums. In this paper, we extend this approach by considering intervals as integration limits instead of real numbers and we abolish the inclusion monotonicity restriction of the interval functions and this notion is used to determine interval probability density functions.
\end{abstract}

\paragraph{[Keywords:} Interval Mathematics, Moore and Yang integrals,
Riemann sums, interval probability density functions.

\section{Introduction}

The classical probability theory provides a mathematical model for the study of uncertainties of a random nature, also called uncertain knowledge that, even if produced identically, presents to each experiment, results that vary unpredictably. In it, random events must be precisely defined, which is often not the case in real situations. Therefore, in the classical theory of probabilities consider an absolute knowledge  without consider  the uncertainties generated from incomplete or imprecise knowledge present  in many real situations and to deal with these uncertainties,  were proposal several ways to extend the notion of probabilities  \cite{Asmus17,Costa13,Fat19,Kre04,Wei00,Yuan17}.
Within this context, many researchers have developed studies on inaccurate probabilities, in order to consider problems that are not contemplated by classical theory \cite{Jab16,Li12}. Among them, the interval mathematics has served as a theoretical framework for this purpose as one can see on the works Yager \cite{Yager84},  Sarveswaram et al. \cite{Sarv98},   Tanaka and Sugihara \cite{Tan04}, Intan \cite{Intan07}, Zhang, et al. \cite{Zhang09},  Campos and dos Santos \cite{Cam13}, and Jamison and Lodwick \cite{JL20}.

In this work we are interested in studying the interval case of the probability density functions constructed from integrals. Also known as density of a 
continuous random variable, these functions describe the probability of a random variable assuming that the value of the random variable would equal that 
sample the given value in the sample space, which can be calculated from the integral of the density of that variable in a given range, i.e. given a 
continuous variable $X$ and the values set $X(S)$, the probability density function $f$ asigns to each element $x \in X(S)$ a number $f(x)$ satisfying 
properties: (1) $f(x) > 0$ and (2) $\int f(x) = 1$. In this framework some researchers have been studied the interval version of probability density 
functions in order to provide a model where the uncertainty of the varariable is measured by an interval. For instance, Berleant in \cite{DB93} has 
developed an system to evaluate the reasoning automatically by using an intervals probability density functions. Ram\'irez \cite{PR05} used a data 
sampling interval to analyze its influence in the estimation of the parameters of the Weibull wind speed probability density distribution. In both cases 
authors has been used interval to estimate the uncertainty considering the usual version of probability density  function, i.e. 
they not consider an interval version 
of probability density function indeed. Here, we propose an interval probability density 
interpretation based on a new way to define interval integrals.

In \cite{MY59} Moore and Yang defined the first approach to
interval integrals. In \cite{MSY60} further properties of this
interval integral notions and  their relation with integration of
ordinary real functions are established. The Moore-Yang
integration of interval functions approach is based on a
generalization of Riemann sums for real integrals. Moore in
\cite{Moo79} defined integrals of continuous function from $\RR$
into $\II(\RR)$ which can be seen as restrictions of continuous
and monotonic interval functions to degenerate intervals (this is
guarantee by the dependency on the degenerate intervals of
interval integral \cite{MSY60}).

There exist some other few interval integration approaches in the
literature. Among other, are the work of Caprani, Madsen and Rall
in \cite{CMR81}, they defined  integrals for  functions (not
necessarily continuous) from an interval to the set of interval of
extended real numbers (i.e. real numbers with $-\infty$ and
$+\infty$) based on Darboux integrals. The Caprani, Madsen and
Rall approach have Moore integrals (as defined in \cite{Moo79}) as
an special case. Lately,  Rall in \cite{Ral82}, considering this
interval integration approach, partially solved the problem of
assignment of infinite intervals to some improper intervals known
to be finite. Corliss in \cite{Cor87} extend the Caprani, Madsen
and Rall  notion of integral for considering interval valued
limits. More recently, we can cite the following works on integration of interval-valued functions \cite{Bus13,Cha15,JL20}.

 The Escard\'o approach in
\cite{Esc97} adapted the Edalat work \cite{Eda95}, who extended
the Riemann integrals using Scott domain theory, for the domain of
intervals, as seen in \cite{Sco70}, resulting in a new
version of the Moore-Yang integrals, which consider Scott
continuous functions instead of continuous (w.r.to Moore metrics)
and inclusion monotonic functions (notice that all Scott
continuous interval functions are inclusion monotonic). The
relationship between both continuity notion can be found in
\cite{AB97,Bed13,SBA06}.

The Moore-Yang integrals have two restrictions:
\begin{enumerate}
\item The interval function must be inclusion monotonic and
\item The limits of integration are real numbers.
\end{enumerate}
Moore, Strother and Yang in \cite{MSY60} (pp. A-5) asseverate
that:

``Theorem 4 ({\em theorem ~\ref{teo-moo-char} in this paper})
suggests a more general definition for $\int F$ may be feasible --
namely the right side of the equality in theorem 4. This
conclusion could lead to a deletion of the condition $A\subset
B\Rightarrow F(A)\subset F(B)$.''

So,  the first restriction, is not necessary and can be suppressed
considering a more general definition of the integrals based on a
characterization of the interval integral in term of an interval
of  real integrals. But, taken this definitions as primitive
implies  that the integral interval notion is not a generalization
of a real integral approach, resulting in a notion with a poor
mathematical foundation.

This paper define an integral for interval functions which extend
the Moore-Yang approach eliminating both restrictions. We also
give a characterization of this extension in terms of the extremes
of the limits of integration which could simplify its computation.

Finally, we use this integral notion to determine probability density functions \cite{Eva00} in the context of interval probability 
\cite{Cam00,Lod08,Tes92}.

\section{Preliminary}

Let $\II(\RR)$ be the set of real closed intervals, or simply
intervals.  We will use upper letters at start of alphabet to
indicate an interval. The left and right extremes of an interval
$A$ will be denoted by $\underline{a}$ and $\overline{a}$
respectively, thus $A=[\underline{a},\overline{a}]$. We define two
projections for intervals: $\pi_1(A)=\underline{a}$ and
$\pi_2(A)=\overline{a}$.

The partial order that we will use for intervals will be the
Kulisch-Miranker one \cite{KM81}, i.e.

$$A\leq B\Leftrightarrow \underline{a}\leq \underline{b}\mbox{ and
}\overline{a}\leq\overline{b}.$$

Particularly,

$$A\ll B\Leftrightarrow \underline{a} < \underline{b}\mbox{ and
}\overline{a}< \overline{b}.$$

The partial order used in the Moore and Yang integral approach was
 the inclusion of sets, i.e.

$$A\subseteq B\Leftrightarrow \underline{b}\leq \underline{a}
\leq \overline{a}\leq\overline{b}.$$

The arithmetical operations on intervals are defined as follow

$$A\oplus B=\{x\oplus y\;\mid\; x\in  A\mbox{ and }y\in B\},$$

\noindent where $\oplus$ is any one of the usual  arithmetical
operations. The unique restriction is that in the case of
division, $B$ can not contain $0$. Each operations can be
characterized in terms of their extremes as follows:

$$A\oplus B=[\min\{\underline{a}\oplus \underline{b},\underline{a}\oplus
\overline{b},\overline{a}\oplus \underline{b} ,\overline{a}\oplus
\overline{b}\}, \max\{\underline{a}\oplus
\underline{b},\underline{a}\oplus \overline{b},\overline{a}\oplus
\underline{b} ,\overline{a}\oplus \overline{b}\}].$$

For an abuse of language, given a real number $r$, we will write $A\oplus r$ and $r\oplus A$ instead of $A\oplus [r,r]$ and $[r,r]\oplus A$, respectively.

In the case of addition, this expression can be abbreviated as

$$A+B=[\underline{a}+\underline{b},\overline{a}+\overline{b}].$$

If either $\underline{a}\geq 0$ or $\underline{b}\geq 0$ then
$$AB=[\underline{a}\underline{b},\overline{a}\overline{b}].$$

 Functions whose domain and co-domain  are subsets of  $\II(\RR)$
are called interval functions. Let $F: \mathcal{I}\lra \II(\RR)$
be an interval function. Define the functions $\underline{F}:
\mathcal{I}\lra \RR$ and $\overline{F}: \mathcal{I}\lra \RR$ by
$\underline{F}(A)=\pi_1(F(A))$ and $\overline{F}(A)=\pi_2(F(A))$.
Trivially, $F(A)=\left [\underline{F}(A), \overline{F}(A)\right]$.
Sometimes $F$ will be denoted by $[\underline{F},\overline{F}]$.
An interval function $F$ is said to be an inclusion monotonic
function if it is monotonic w.r.to the inclusion, i.e. $A\subseteq
B\Ra F(A)\subseteq F(B)$.

A distance between two intervals is defined by:

$$d_M(A,B)=max\{\mid \underline{b}-\underline{a}\mid,
\mid\overline{b}-\overline{a}\mid\}.$$

In \cite{Moo62} was proved that $d_M$ is a metric. So, an interval
function $F$ is continuous, if it is continuous w.r.to the metric
$d_M$.

\section{Moore Approach}

In this section we will overview the main definitions of the Moore
and Yang interval integral approach.

\begin{definition}
Let $A$ be a real interval. A {\bf partition} of $A$ is a sequence
  $\mathcal{T}=\{\underline{a}=x_0,x_1,\ldots
,x_n=\overline{a}\}$ such that for each $i=0,\ldots,n-1$, $x_i<
x_{i+1}$. The set of all partition of $A$ will be denoted by
$\frak{T}(A)$.
\end{definition}

\begin{definition} Let $A$ be a real interval. A partition $\mathcal{T}_1$
of $A$ {\bf is finer} than the partition $\mathcal{T}_2$ of $A$,
denoted by $\mathcal{T}_1\preceq \mathcal{T}_2$,  if
$\mathcal{T}_2\subseteq \mathcal{T}_1$.
\end{definition}

Clearly, $\preceq$ is a partial order on $\frak{T}(A)$.

\begin{proposition} \label{pro-part-lat1}
Let $A$ be a real interval. $\prt{\frak{T}(A),\preceq}$ is a
lattice with greatest  element.
\end{proposition}

\proof{ Let $\mathcal{T}_1$ and $\mathcal{T}_2$ be partitions of
$A$. Then trivially, $\mathcal{T}_1\cap \mathcal{T}_2$ and
$\mathcal{T}_1\cup \mathcal{T}_2$ are the {\it supremum} and {\it
infimum}, respectively,  of $\mathcal{T}_1$ and $\mathcal{T}_2$.
The greatest element of $\frak{T}(A)$ is the partition
$\{\underline{a},\overline{a}\}$. \hfill \small $\blacksquare$}

\begin{definition}
Let $A$  be an interval and $F$ be an inclusion monotonic
continuous interval function. The {\bf Riemann sum} of $F$ w.r.to
a partition $\mathcal{T}$ of $A$ is defined by:

$$\sum(F,\mathcal{T})=
{\displaystyle \sum_{k=0}^{n-1}}F([x_k,x_{k+1}])d(x_k,x_{k+1}),$$

\noindent where $d$ is the usual metric on the real numbers, i.e.
$d(x,y)=\brs{x-y}$.

 The Moore and Yang integral of $F$ at the
interval $A$ is defined by

$$\int_A F(X)dX=\bigcap_{\mathcal{T}\in \frak{T}(A)}
\sum(F,\mathcal{T}).$$
\end{definition}

\begin{theorem} (Characterization theorem) \label{teo-moo-char}
Let $A$  be an interval and $F$ be an inclusion monotonic
continuous interval function. Then

$$\int_A F(X)dX=\left [ \int_{\underline{a}}^{\overline{a}} f_l(x)dx, \int_{\underline{a}}^{\overline{a}} f_r(x)dx\right ],$$

\noindent where $f_l(x)=\pi_1 F[x,x]$ and $f_r(x)=\pi_2 F[x,x]$.
\end{theorem}

\proof{See \cite{MSY60}. \hfill \small $\blacksquare$}

\section{Our Extension}

\begin{definition}
Let $A$ and $B$ be two  real intervals such that $A\rojo{\ll} B$.
Define,

$$\mathcal{I}_{[A,B]}=\{X\in\II(\RR)\;\mid\; A\leq X\leq B\}$$

\noindent and $\overline{AB}=\left \{X\;\mid\;
\underline{x}=x\mbox{ and
}\overline{x}=\overline{a}+\frac{\displaystyle
\overline{b}-\overline{a}}{\displaystyle
\underline{b}-\underline{a}}(x-\underline{a})\mbox{ for some }x\in
[\underline{a},\underline{b}]\right \}$.
\end{definition}

Since $A\ll B$, then $\underline{a}<\underline{b}$ and
$\overline{a}<\overline{b}$. In what follows, without lost of
generality, we will suppose that $\underline{a}<\underline{b}$.

Given a metric space $(\mathcal{X},d)$ and a subset $A$ of $\mathcal{X}$ we define the
diameter of $A$, denoted by $diam(A)$, by

$$diam(A)=\sup\{d(a,b)\mid a,b\in A\}.$$

\begin{definition} \label{bounded_def}
A subset $A$ of a metric space $(\mathcal{X},d)$ is {\bf bounded} if
$diam(A)$ is finite
\end{definition}

\begin{definition}
A function of a non-empty set into a metric space is called a {\bf
bounded function} if its image is a bounded set.
\end{definition}

Clearly a subset $\mathcal{I}$ of $\II(\RR)$ is bounded if, and
only if, it is contained in $\mathcal{I}_{[A,B]}$ for some $A,B\in
\II(\RR)$ such that $A\ll B$. Thus, an interval function $F:\mathcal{I}\lra\II(\RR)$,
where $\mathcal{I}\subseteq\II(\RR)$, is bounded if, and only if,
its image is contained in $\mathcal{I}_{[A,B]}$ for some $A,B\in
\II(\RR)$. That is, if $A\leq F(X)\leq B \;\forall X\in
\mathcal{I}$. Therefore, the topological and order based notions
of boundedness for an interval function
$F:\mathcal{I}\lra\II(\RR)$ coincides.

 Let $\phi :\II(\RR)\lra{\RR}^2$ be the function defined by,
 $\phi ([a,b])=(a,b)$. Clearly
the function $\phi$ is continuous and injective. Thus, we can
identify $\II(\RR)$ with a subspace of $\RR^2$, where $\RR^2$ is
considered here with the metric,

$$d((\underline{a}, \overline{a}), (\underline{b}, \overline{b}))=max\{\mid \underline{b}-\underline{a}\mid,
\mid\overline{b}-\overline{a}\mid\}.$$

This is not restrictive since usual metrics in $\RR^n$ are
equivalents, from the topological point-of-view.

Notice that $\mathcal{I}_{[A,B]}$ is closed and bounded in
$\RR^2$, hence, by the Heine-Borel theorem \cite{Sim63} p. 119, it
is compact.

\begin{lemma}(Completeness Lemma) \label{lemma_comp}
Let $X$ be a bounded subset of $\II(\RR)$. Define
$\underline{X}=\{\underline{c}\; \mid\; C\in X\}$ and
$\overline{X}=\{\overline{c}\; \mid\; C\in X\}$. Then, $\bigsqcup
X=\left [\bigsqcup \underline{X},\bigsqcup \overline{X}\right ]$
and $\bigsqcap X=\left [\bigsqcap \underline{X},\bigsqcap
\overline{X}\right ]$ where $\bigsqcap \underline{X} = \inf \underline{X}$ and $\bigsqcup \underline{X} = \sup \underline{X}$.
\end{lemma}

\proof{ If $X$ is a bounded subset of $\II(\RR)$ then $X\subseteq
\mathcal{I}_{[A,B]}$ for some $A,B\in \II(\RR)$. Clearly
$\underline{a}$ is a lower bound of $\underline{X}$ and
$\overline{X}$ and $\overline{b}$ is an upper bound of
$\underline{X}$ and $\overline{X}$. Since each upper bounded
subset of the real numbers has a {\it supremum} and each lower
bounded subset of the real numbers has an {\it infimum}, then
$\underline{X}$ and $\overline{X}$ has {\it supremum} and {\it
infimum}. It is easy to see that $\bigsqcup X=\left [\bigsqcup
\underline{X},\bigsqcup \overline{X}\right ]$ and $\bigsqcap
X=\left [\bigsqcap \underline{X},\bigsqcap \overline{X}\right ]$.
\hfill \small $\blacksquare$}

\begin{corollary}
Let $\mathcal{I}\subseteq\II(\RR)$ and $F:\mathcal{I}\lra
\II(\RR)$ be a bounded interval function. Then, the set
$F(\mathcal{I})=\{F(X)\in \II(\RR)\mid X\in \mathcal{I}\}$ has
{\it supremum} and {\it infimum}. In fact, if
$F(X)=[\underline{F}(X), \overline{F}(X)]$ then,

$$\bigsqcup F(\mathcal{I})=\left [\bigsqcup \underline{F}(\mathcal{I}), \bigsqcup\overline{F}(\mathcal{I})\right ]$$

\noindent and

$$\Bigsqcap\; F(\mathcal{I})=\left [\Bigsqcap \underline{F}(\mathcal{I}), \Bigsqcap
\overline{F}(\mathcal{I})\right ].$$
\end{corollary}

\proof{ Since $F$ is a bounded interval function we have that
$F(\mathcal{I})\subseteq \mathcal{I}_{[A,B]}$ for some $A,B\in
\II(\RR)$. The statements follows from the Completeness lemma, by
taking $X=F(\mathcal{I})$. \hfill \small $\blacksquare$}

\begin{lemma}\label{lemma-sup}
Let $\mathcal{I}$ and $\mathcal{J}$ be bounded subsets of
$\II(\RR)$. Let

$$\mathcal{I} +\mathcal{J}=\{X+Y\mid X\in \mathcal{I}\;,Y\in\mathcal{J}\}.$$

Then, $\bigsqcap (\mathcal{I} +\mathcal{J})=\bigsqcap
(\mathcal{I}) +\bigsqcap(\mathcal{J})$ and $\bigsqcup (\mathcal{I}
+\mathcal{J})=\bigsqcup (\mathcal{I}) +\bigsqcup(\mathcal{J})$
\end{lemma}

\proof{ The result  follows by the proof of Lemma \ref{lemma_comp}, by
the characterization of the interval addition and by general
properties of {\it supremum} and {\it infimum}. \hfill \small
$\blacksquare$}

\begin{definition}
A set $\mathcal{P}$ is a {\bf partition} of $\mathcal{I}_{[A,B]}$
if there exists partitions $\mathcal{T}_1$ and $\mathcal{T}_2$ of
$[\underline{a},\underline{b}]$ and $[\overline{a},\overline{b}]$,
respectively, such that $\mathcal{P}=
(\mathcal{T}_1\times\mathcal{T}_2)\cap\II(\RR)$.
\end{definition}

\begin{remark} Sometimes we say that the partition $\mathcal{P}$
of $\mathcal{I}_{[A,B]}$ cames from the partitions $\mathcal{T}_1$
and $\mathcal{T}_2$ of $[\underline{a},\underline{b}]$ and
$[\overline{a},\overline{b}]$, respectively, or simply, that
$\mathcal{P}$ of $\mathcal{I}_{[A,B]}$ cames from
$\mathcal{T}_1\times\mathcal{T}_2$.
\end{remark}

\begin{definition}
We say that a partition $\mathcal{P}'$ of $\mathcal{I}_{[A,B]}$,
 is {\bf finer than} the partition $\mathcal{P}$
of $\mathcal{I}_{[A,B]}$,  denoted by $\mathcal{P}'\preccurlyeq
\mathcal{P}$, if $\mathcal{P}\subseteq \mathcal{P}'$.
\end{definition}

\begin{lemma} \label{lemma-part}
Let $\frak{P}[A,B]=\{\mathcal{P}\;\mid\;\mathcal{P}$ is a
partition of $\mathcal{I}_{[A,B]}\}$. Then
$\prt{\frak{P}[A,B],\preccurlyeq}$ is a lattice with greatest
element.
\end{lemma}

\proof{ Let $\mathcal{P}$ and $\mathcal{P}'$ be partitions of
$\mathcal{I}_{[A,B]}$ coming from
$\mathcal{T}_1\times\mathcal{T}_2$ and
$\mathcal{T}'_1\times\mathcal{T}'_2$, respectively. Then
trivially, from definition of $\preccurlyeq$ and proposition
~\ref{pro-part-lat1} we have that $\mathcal{P}\sqcup\mathcal{P}'$
is the partition coming from
$\mathcal{T}_1\times\mathcal{T}_2\cap\mathcal{T}'_1\times\mathcal{T}'_2$.
Analogously, we have that $\mathcal{P}\sqcap\mathcal{P}'$ is the
partition coming from
$\mathcal{T}_1\times\mathcal{T}_2\cup\mathcal{T}'_1\times\mathcal{T}'_2$.
The greatest element of $\frak{P}[A,B]$ is the partition
$\mathcal{P}_\top=\{A,B\}$ which came from
$\{\underline{a},\underline{b}\}\times
\{\overline{a},\overline{b}\}$. \hfill \small $\blacksquare$}

\begin{definition}
Let $F:\mathcal{I}_{[A,B]}\lra \II(\RR)$ be a bounded interval
function. Given a partition $\mathcal{P}$ of
$\mathcal{I}_{[A,B]}$, we define the following {\bf Riemann sums}
of $F$ w.r.to $\mathcal{P}$ namely:

\begin{itemize}
\item Lower Riemann sum --

$$\sigma (F,\mathcal{P})=\sum_{i=0}^{n-1} \sum_{j=0}^{m-1} \Bigsqcap\; F(\mathcal{I}_{[P_{i,j},P_{i+1,j+1}]}) d_M(P_{i,j},P_{i+1,j+1})
 \mbox{ and}$$

 \item Upper Riemann sum --

$$\sum (F,\mathcal{P})=\sum_{i=0}^{n-1} \sum_{j=0}^{m-1} \bigsqcup\; F(\mathcal{I}_{[P_{i,j},P_{i+1,j+1}]}) d_M(P_{i,j},P_{i+1,j+1})
 $$
\end{itemize}

\noindent where $P_{i,j}=[x_i,y_j]$ and, as usual,
$F(\mathcal{I}_{[P_{i,j},P_{i+1,j+1}]})= \left\{F(X) \;\mid\;
X\in\mathcal{I}_{[P_{i,j},P_{i+1,j+1}]}\right\}$.
\end{definition}

\begin{lemma} \label{lemma-riemman1}
Let $F:\mathcal{I}_{[A,B]}\lra \II(\RR)$ be a bounded function and
let $\mathcal{P}$ be a partition of $\mathcal{I}_{[A,B]}$. Then
$$\sigma(F,\mathcal{P}) \leq \sum (F,\mathcal{P}).$$
\end{lemma}

\proof{ Clearly, for each $i=0,\ldots, n-1$ and $j=0,\ldots,m-1$,
$\bigsqcap F(\mathcal{I}_{[P_{i,j},P_{i+1,j+1}]})\leq \bigsqcup
F(\mathcal{I}_{[P_{i,j},P_{i+1,j+1}]})$ and
$d_M(P_{i,j},P_{i+1,j+1})>0$. So,

$$\Bigsqcap\;
F(\mathcal{I}_{[P_{i,j},P_{i+1,j+1}]})d_M(P_{i,j},P_{i+1,j+1})\leq
\bigsqcup
F(\mathcal{I}_{[P_{i,j},P_{i+1,j+1}]})d_M(P_{i,j},P_{i+1,j+1}).$$

Therefore,

$$\sigma(F,\mathcal{P}) \leq \sum (F,\mathcal{P}).$$

\vspace{-.5cm} \hfill \small $\blacksquare$}

\begin{proposition} \label{lemma-riemman2}
Let $F:\mathcal{I}_{[A,B]}\lra \II(\RR)$ be a bounded function and
let $\mathcal{P}'$ and $\mathcal{P}$ be partitions of
$\mathcal{I}_{[A,B]}$. If $\mathcal{P}'\preccurlyeq \mathcal{P}$
then

$$\sigma(F,\mathcal{P})\leq \sigma (F,\mathcal{P}')\leq
\sum (F,\mathcal{P}')\leq \sum (F,\mathcal{P}).$$
\end{proposition}

\proof{Let $\mathcal{P}$ be a partition cames from
$\mathcal{T}_1\times \mathcal{T}_2$ where
$\mathcal{T}_1=\{\underline{a}=x_0,\ldots,x_n=\underline{b}\}$.
With no lost of generality we may assume that the partition
$\mathcal{P}'$ cames from $\mathcal{T}'_1\times \mathcal{T}_2$
where
$\mathcal{T}'_1=\{\underline{a}=x_0,\ldots,x_{i-1},z,x_{i+1},\ldots,x_n=\underline{b}\}$.

Let $Q_j=[z,y_j]$ then

$$d_M(P_{i-1,j}, P_{i,j+1})\leq d_M(P_{i-1,j},Q_j)+d_M(Q_{j+1}, P_{i,j+1}).$$

On the other hand, we have that

$$\Bigsqcap\; F(\mathcal{I}_{[P_{i-1,j}, P_{i,j+1}]})\leq \Bigsqcap\;
F(\mathcal{I}_{[P_{i-1,j},Q_j]})\mbox{ and }\Bigsqcap\;
F(\mathcal{I}_{[P_{i-1,j}, P_{i,j+1}]})\leq  \Bigsqcap\;
F(\mathcal{I}_{[Q_{j+1}, P_{i,j+1}]}).$$

Therefore, we have that

$$\begin{array}{rcl}
\sigma (F,\mathcal{P}') & = & \Sumat{r=1}{i-1}
\Sumat{s=0}{m-1}\Bigsqcap\; F(\mathcal{I}_{[P_{r-1,s},
P_{r,s+1}]})d_M(P_{r-1,s}, P_{r,s+1})+ \\
& & \Sumat{s=0}{m-1}
\Bigsqcap\; F(\mathcal{I}_{[P_{i-1,s}, Q_{s+1}]})d_M(P_{i-1,s}, Q_{s+1})+ \\
& &  \Sumat{s=0}{m-1} \Bigsqcap\;
F(\mathcal{I}_{[Q_{s},P_{i,s+1}]})d_M(P_{i,s+1}, Q_{s}) + \\
& & 
\Sumat{r=i}{n-1} \Sumat{s=0}{m-1}\Bigsqcap\;
F(\mathcal{I}_{[P_{r,s},
P_{r+1,s+1}]})d_M(P_{r,s}, P_{r+1,s+1}) \\
\\
 & \geq & \Sumat{r=1}{i-1} \Sumat{s=0}{m-1}\Bigsqcap\;
F(\mathcal{I}_{[P_{r-1,s}, P_{r,s+1}]})d_M(P_{r-1,s}, P_{r,s+1})+ \\
& & 
\Sumat{s=0}{m-1} \Bigsqcap\; F(\mathcal{I}_{[P_{i-1,s},
P_{i,s+1}]})d_M(P_{i-1,s}, Q_{s+1})+  \\
& & \Sumat{s=0}{m-1} \Bigsqcap\;
F(\mathcal{I}_{[P_{i-1,s},P_{i,s+1}]})d_M(P_{i,s+1}, Q_{s}) + \\
& & 
\Sumat{r=i}{n-1} \Sumat{s=0}{m-1}\Bigsqcap\;
F(\mathcal{I}_{[P_{r,s},
P_{r+1,s+1}]})d_M(P_{r,s}, P_{r+1,s+1}) \\
\end{array} $$

The second inequality was proved in Lemma ~\ref{lemma-riemman1}
and the third inequality follows by the same token of the first.
\hfill \small $\blacksquare$}

\begin{corollary}\label{coro-riemman3}
Let $F:\mathcal{I}_{[A,B]}\lra \II(\RR)$ be a bounded function and
let $\mathcal{P}$ and $\mathcal{Q}$ be partitions of
$\mathcal{I}_{[A,B]}$. Then,

$$\sigma(F,\mathcal{P})\leq \Sigma(F,\mathcal{Q}).$$
\end{corollary}

\proof{ Since the partition $P\cup Q$ refines $P$ and $Q$ we have
that

$$\sigma(F,\mathcal{P})\leq \sigma(F,\mathcal{P\cup Q})\leq \Sigma(F,\mathcal{P\cup
Q})\leq \Sigma(F,\mathcal{Q}).$$

\vspace{-.5cm} \hfill \small $\blacksquare$}

\begin{definition}
Let $F:\mathcal{I}_{[A,B]}\lra \II(\RR)$ be a bounded function. We
define the {\bf lower integral} of $F$ w.r.to $A$ and $B$, denoted
by ${\displaystyle \underline{\int_A^B}F(X)dX}$, by

$${\displaystyle \underline{\int_A^B}
F(X)dX=\bigsqcup_{\mathcal{P}\in\frak{P}[A,B]} \sigma
(F,\mathcal{P})}$$

\noindent  and  the {\bf upper integral} of $F$ w.r.to $A$ and
$B$, denoted by ${\displaystyle \overline{\int_A^B}F(X)dX}$, by

$${\displaystyle
\overline{\int_A^B}F(X)dX=\Bigsqcap_{\mathcal{P}\in\frak{P}[A,B]}
\sum (F,\mathcal{P})}.$$
\end{definition}

\begin{proposition}
Let $F:\mathcal{I}_{[A,B]}\lra \II(\RR)$ be a bounded function
such that   $C\leq F(X)\leq D\;\forall X\in \mathcal{I}_{[A,B]}$.
Then, for any partition $\mathcal{P}$ of $\mathcal{I}_{[A,B]}$ we
have that
$$Cd_M(A,B)\leq \sigma(F,\mathcal{P})\leq
\underline{\int_A^B}F(X)dX\leq \overline{\int_A^B}F(X)dX \leq
\sum(F,\mathcal{P})\leq Dd_M(A,B).$$
\end{proposition}

\proof{ Let $\mathcal{P}_{\top}=\{A,B\}$ be the trivial partition
of $\mathcal{I}_{[A,B]}$. Then, by Lemma \ref{lemma-riemman2} we
have that
$\sigma(F,\mathcal{P}_{\top})\leq\sigma(F,\mathcal{P})$. But, by
definition, we have that $\sigma(F,\mathcal{P}_{\top})=\bigsqcap
F(\mathcal{I}_{[A,B]})d_M(A,B)\geq Cd_M(A,B)$, which proves the
first inequality. The last inequality follows by the same token.

The inequality ${\displaystyle \underline{\int_A^B}F(X)dX\leq
\overline{\int_A^B}F(X)dX }$ follows by Corollary
~\ref{coro-riemman3}. The remaining inequalities follows by
definition of the lower and upper integrals. \hfill \small
$\blacksquare$}

\begin{proposition}
Let $\frak{P'}$ and $\frak{P''}$ be subset of $\frak{P}[A,B]$
satisfying the following properties:

$$\begin{array}{ll}
(\bullet) & \forall \mathcal{P}\in \frak{P}[A,B]\; \exists
\mathcal{P'}\in \frak{P'}  \mbox{ and }\mathcal{P''}\in
\frak{P''}\mbox{ such that }\sigma(F,\mathcal{P})\leq
\sigma(F,\mathcal{P'}) \mbox{ and } \\
& \sum (F,\mathcal{P''})\leq \sum(F,\mathcal{P}). \end{array}$$


Then,

$$\underline{\int_A^B}
F(X)dX=\bigsqcup_{\mathcal{P'}\in\frak{P'}} \sigma
(F,\mathcal{P'})$$

\noindent and

$$\overline{\int_A^B}
F(X)dX=\Bigsqcap_{\mathcal{P''}\in\frak{P''}} \sum
(F,\mathcal{P''}).$$
\end{proposition}

\proof{ This follows by general properties of {\it supremum} and
{\it infimum}. \hfill \small $\blacksquare$}

\begin{corollary}\label{coro-part2}
Let $A\ll B$, $C\in \overline{AB}$
and let $\frak{P'}[A,B]$ be the subset of $\frak{P}[A,B]$
consisting of partition containing $C$. Then,

$$\underline{\int_A^B}
F(X)dX=\bigsqcup_{\mathcal{P}\in\frak{P'}[A,B]} \sigma
(F,\mathcal{P})$$

\noindent and

$$\overline{\int_A^B}
F(X)dX=\Bigsqcap_{\mathcal{P}\in\frak{P'}[A,B]} \sum
(F,\mathcal{P}).$$
\end{corollary}

\proof{ From a partition $\mathcal{P}\in\frak{P}[A,B]$ build the
partition $\mathcal{P'}=\mathcal{P}\cup \{C\}$, containing $C$.
Since $\mathcal{P'}$ is finer than $\mathcal{P}$ we have that
$\sigma(F,\mathcal{P})\leq\sigma(F,\mathcal{P'})$ and
$\sum(F,\mathcal{P'})\leq\sum (F,\mathcal{P})$. Thus,
$\frak{P'}[A,B]$ satisfies the condition $(\bullet)$ of the above
proposition. \hfill \small $\blacksquare$}

\begin{corollary}
Let $\mathcal{P}_n=\{A=X_0,X_1,\ldots, X_n=B\}$ be the partition
of $\mathcal{I}_{[A,B]}$ coming from the partitions
$\mathcal{T}=\{\underline{a}=t_0,t_1,\ldots,t_n=\underline{b}\}$ and $\mathcal{S}=\{\overline{a}=s_0,s_1,\ldots,s_n=\overline{b}\}$
of $[\underline{a}, \underline{b}]$ and $[\overline{a}, \overline{b}]$, where
$t_k=\underline{a}+\frac{\displaystyle k}{\displaystyle
n}(\underline{b}-\underline{a})$ and $s_k=\overline{a}+\frac{\displaystyle k}{\displaystyle
n}(\overline{b}-\overline{a})$, respectively. Then,

$$\underline{\int_A^B}
F(X)dX=\bigsqcup_{n\in \NN} \sigma (F,\mathcal{P}_n)$$

\noindent and

$$\overline{\int_A^B} F(X)dX= \Bigsqcap_{n\in \NN} \sum
(F,\mathcal{P}_n).$$
\end{corollary}

\proof{ Let $\mathcal{P}\in\frak{P}[A,B]$ be a partition of
$\mathcal{I}_{[A,B]}$. Clearly, $\exists n\in \NN$ such that
   $\mathcal{P}_n\preceq \mathcal{P}$. Therefore,
$\sigma(F,\mathcal{P})\leq\sigma(F,\mathcal{P}_n)$ and
$\sum(F,\mathcal{P}_n)\leq\sum (F,\mathcal{P})$. Thus,
$\frak{P}=\{\mathcal{P}_n\mid n\in \NN\}$ satisfies the condition
$(\bullet)$ \hfill \small $\blacksquare$}


%
%
%

%
%

\begin{definition}
A bounded function $F:\mathcal{I}_{[A,B]}\lra \II(\RR)$ is said to
be an {\bf integrable function} if

$${\displaystyle
\underline{\int_A^B}F(X)dX=\overline{\int_A^B}F(X)dX}.$$

 This common value is called the {\bf interval integral} of $F$ w.r.to
$A$ and $B$ and it is denoted by ${\displaystyle \int_A^B
F(X)dX}$.
\end{definition}

\begin{definition}
Let $F:\mathcal{I}_{[A,B]}\lra \II(\RR)$ be a bounded function.
Define the {\bf left and right spectrum} of $F$, denoted by $F_l$
and $F_r$ respectively, by

$$F_l(x)=\pi_1F\left [x,\overline{a}+\frac{\displaystyle
\overline{b}-\overline{a}}{\displaystyle
\underline{b}-\underline{a}}(x-\underline{a})\right ]$$

\noindent and

$$F_r(x)=\pi_2F\left [x,\overline{a}+\frac{\displaystyle
\overline{b}-\overline{a}}{\displaystyle
\underline{b}-\underline{a}}(x-\underline{a})\right ],$$

\noindent where $\pi_1$ and $\pi_2$ are the left and right
projections from $\II(\RR)$ to $\RR$ and $x\in [\underline{a},
\underline{b}]$.
\end{definition}

\begin{theorem}
Let $F:\mathcal{I}_{[A,B]}\lra \II(\RR)$ be a continuous function.
Then,

$$\underline{\int_A^B}F(X)dX=\left
[\underline{\int_{\underline{a}}^{\underline{b}}}F_l(x)dx,
 \underline{\int_{\underline{a}}^{\underline{b}}}F_r(x)dx\right
]\frac{d_M(A,B)}{\underline{b}-\underline{a}}$$

\noindent and

$$\overline{\int_A^B}F(X)dX=\left
[\overline{\int_{\underline{a}}^{\underline{b}}}F_l(x)dx,
 \overline{\int_{\underline{a}}^{\underline{b}}}F_r(x)dx\right
]\frac{d_M(A,B)}{\underline{b}-\underline{a}}.$$
\end{theorem}

\proof{ We will only prove the first equality since the second one
follows analogously.

 Let $\mathcal{P}_n=\{A=X_0,X_1,\ldots, X_n=B\}$ be
the partition of $\mathcal{I}_{[A,B]}$ coming from the partitions
$\mathcal{T}=\{\underline{a}=t_0,t_1,\ldots,t_n=\underline{b}\}$ and $\mathcal{S}=\{\overline{a}=s_0,s_1,\ldots,s_n=\overline{b}\}$
of $[\underline{a}, \underline{b}]$ and $[\overline{a}, \overline{b}]$, where
$t_k=\underline{a}+\frac{\displaystyle k}{\displaystyle
n}(\underline{b}-\underline{a})$ and $s_k=\overline{a}+\frac{\displaystyle k}{\displaystyle
n}(\overline{b}-\overline{a})$, respectively.

of $[\underline{a}, \underline{b}]$, where
$t_k=\underline{a}+\frac{\displaystyle k}{\displaystyle
n}(\underline{b}-\underline{a}).$ Then,

$$d_M(X_k, X_{k+1})=\frac{\displaystyle d_M(A,B)}{\displaystyle n}.$$

Therefore,

$$\begin{array}{rcl} \underline{\displaystyle \int_A^B} F(X)dX & = &
{\displaystyle  \bigsqcup_{n\in\NN}\sigma(F,\mathcal{P}_n)}\\
                            & = & {\displaystyle \bigsqcup_{n\in\NN} \sum_{k=0}^{n-1}}
                            \Bigsqcap\;
                            F(\mathcal{I}_{[X_k,X_{k+1}]})
                            d_M(X_{k+1},X_k)\\
                            & = & d_M(A,B){\displaystyle \bigsqcup_{n\in\NN}\left (\frac{1}{n}\right )
                            \sum_{k=0}^{n-1}}
                            \Bigsqcap\;
                            F(\mathcal{I}_{[X_k,X_{k+1}]})\\
                            & = & d_M(A,B)\left [{\displaystyle
                            \bigsqcup_{n\in\NN}\left (\frac{1}{n}\right )\sum_{k=0}^{n-1}}
                            \Bigsqcap\;
                            \underline{F}(\mathcal{I}_{[X_k,X_{k+1}]}), {\displaystyle
                            \bigsqcup_{n\in\NN}\left (\frac{1}{n}\right )\sum_{k=0}^{n-1}}
                            \Bigsqcap\;
                            \overline{F}(\mathcal{I}_{[X_k,X_{k+1}]})\right
                            ].

\end{array}$$

On the other side, we have that

$$\begin{array}{rcl}
\underline{\displaystyle
\int_{\underline{a}}^{\underline{b}}}F_l(x)dx & = &
{\displaystyle \bigsqcup_{n\in\NN}\sum_{k=0}^{n-1}}F_l(t_k)d(t_k,t_{k+1}) \\
                                                         & = &
{\displaystyle \bigsqcup_{n\in\NN}\sum_{k=0}^{n-1}} F_l(t_k)\left
({\displaystyle \frac{1}{n}} \right )
(\underline{b}-\underline{a}) \\
                                                         & = &
(\underline{b}-\underline{a}){\displaystyle
\bigsqcup_{n\in\NN}\left (\frac{1}{n} \right ) \sum_{k=0}^{n-1}}
\underline{F}(X_k).
\end{array}$$

Analogously we have that

$$\underline{\int_{\underline{a}}^{\underline{b}}}F_r(x)dx=
(\underline{b}-\underline{a}){\displaystyle
\bigsqcup_{n\in\NN}\left (\frac{1}{n} \right ) \sum_{k=0}^{n-1}}
\overline{F}(X_k).$$

Therefore, it is enough to prove that

$${\displaystyle \bigsqcup_{n\in\NN}\left (\frac{1}{n}\right )\sum_{k=0}^{n-1}}
\Bigsqcap\;
\underline{F}(\mathcal{I}_{[X_k,X_{k+1}]})={\displaystyle
\bigsqcup_{n\in\NN}\left (\frac{1}{n}\right ) \sum_{k=0}^{n-1}}
\underline{F}(X_k)$$

\noindent and

$${\displaystyle \bigsqcup_{n\in\NN}\left (\frac{1}{n}\right )\sum_{k=0}^{n-1}}
\Bigsqcap\;
\overline{F}(\mathcal{I}_{[X_k,X_{k+1}]})={\displaystyle
\bigsqcup_{n\in\NN}\left (\frac{1}{n}\right ) \sum_{k=0}^{n-1}}
\overline{F}(X_k).$$

By similarity, we will only prove the former.

Clearly we have that

$$\Bigsqcap\;\underline{F}(\mathcal{I}_{[X_k,X_{k+1}]})\leq \underline{F}(X_k).$$

Thus,

$${\displaystyle \bigsqcup_{n\in\NN}\left (\frac{1}{n}\right )\sum_{k=0}^{n-1}}
\Bigsqcap\; \underline{F}(\mathcal{I}_{[X_k,X_{k+1}]})\leq
{\displaystyle \bigsqcup_{n\in\NN}\left (\frac{1}{n}\right )
\sum_{k=0}^{n-1}} \underline{F}(X_k).$$

Conversely, since $\mathcal{I}_{[X_k,X_{k+1}]}$ is compact and
$\underline{F}$ is a continuous function, we have that $\exists
Y_k\in \mathcal{I}_{[X_k,X_{k+1}]}$ such that $\bigsqcap\;
\underline{F}(\mathcal{I}_{[X_k,X_{k+1}]})=\underline{F}(Y_k)$.

Since $\underline{F}$ is a continuous function and
$\mathcal{I}_{[A,B]}$ is compact we have that $\underline{F}$ is
uniformly continuous.

Let $\epsilon >0$. Since $\underline {F}$ is uniformly continuous,
we have that $\exists \delta >0$ such that if $d_M(X,Y)\leq
\delta$ then $d(\underline{F}(X),\underline{F}(Y))\leq \epsilon$.

Let $n_0\in \NN$ be such that if $n\geq n_0$ then ${\displaystyle
\frac{d_M(A,B)}{n}} \leq \delta$.

Since $Y_k\in \mathcal{I}_{[X_k,X_{k+1}]}$ we have that

$$d_M(X_k,Y_k)\leq d_M(X_k,X_{k+1})=\frac{\displaystyle d_M(A,B)}{\displaystyle n}\leq \delta.$$

Thus, since $\underline {F}$ is uniformly continuous, we have that
$d(\underline{F}(X_k),\underline{F}(Y_k))\leq \epsilon$.
Therefore,

$$\begin{array}{rcl}
{\displaystyle \left (\frac{1}{n}\right
)\sum_{k=0}^{n-1}\underline{F}(X_k)-\left (\frac{1}{n}\right
)\sum_{k=0}^{n-1}} \Bigsqcap\;
\underline{F}(\mathcal{I}_{[X_k,X_{k+1}]}) & = & {\displaystyle
\left (\frac{1}{n}\right )\sum_{k=0}^{n-1}}
[\underline{F}(X_k)-\Bigsqcap\;
\underline{F}(\mathcal{I}_{[X_k,X_{k+1}]}) ] \\
                                           & = &
{\displaystyle \left (\frac{1}{n}\right )\sum_{k=0}^{n-1}}
[\underline{F}(X_k)-\underline{F}(Y_k) ]
\\
& \leq & {\displaystyle \left (\frac{1}{n}\right )\sum_{k=0}^{n-1}} \epsilon \\
& = & \epsilon ,
\end{array}$$

\noindent which proves that

$${\displaystyle \bigsqcup_{n\in\NN}\left (\frac{1}{n}\right )\sum_{k=0}^{n-1}}
\Bigsqcap\;
\underline{F}(\mathcal{I}_{[X_k,X_{k+1}]})={\displaystyle
\bigsqcup_{n\in\NN}\left (\frac{1}{n}\right )
\sum_{k=0}^{n-1}}\underline{F}(X_k).$$

\vspace{-.5cm} \hfill \small $\blacksquare$}

\begin{corollary} \label{coro-charac-gen} [characterization
theorem] If $F:\mathcal{I}_{[A,B]}\lra \II(\RR)$ is a continuous
interval function then, $F$ is an integrable function and

\begin{equation}\label{eq-charac-gen}
\int_A^BF(X)dX=\left
[\int_{\underline{a}}^{\underline{b}}F_l(x)dx,
\int_{\underline{a}}^{\underline{b}}F_r(x)dx\right
]\frac{d_M(A,B)}{\underline{b}-\underline{a}}.
\end{equation}
\end{corollary}

\proof{ Since $F$ is continuous, we have that $F_l$ and $F_r$ are
continuous as well. Therefore,

$$\underline{\int_{\underline{a}}^{\underline{b}}}F_l(x)dx=\overline{\int_{\underline{a}}^{\underline{b}}}F_l(x)dx$$

\noindent and

$$\underline{\int_{\underline{a}}^{\underline{b}}}F_r(x)dx=\underline{\int_{\underline{a}}^{\underline{b}}}F_r(x)dx.$$

Therefore, by the above proposition, we have that

$$\underline{\int_A^B} F(X)dX=\overline{\int_A^B} F(X)dX.$$

Thus, $F$ is an integrable function and

$$\int_A^BF(X)dX=\left
[\int_{\underline{a}}^{\underline{b}}F_l(x)dx,
\int_{\underline{a}}^{\underline{b}}F_r(x)dx\right
]\frac{d_M(A,B)}{\underline{b}-\underline{a}}.$$

\vspace{-.5cm} \hfill \small $\blacksquare$}

\begin{corollary}
If $A=[a,a]$ and $B=[b,b]$ with $a< b$ then

$$\int_A^B F(X)dX = \int_{[a,b]} F(X)dX.$$
\end{corollary}

\proof{ Notice that when $A$ and $B$ are degenerated intervals,
$F_l=f_l$, $F_r=f_r$ and ${\displaystyle
\frac{d_M(A,B)}{\underline{b}-\underline{a}}}=1$. So, this
corollary is straightforward from Corollary ~\ref{coro-charac-gen}
and Theorem ~\ref{teo-moo-char}. \hfill \small $\blacksquare$}

Therefore our approach generalize the Moore and Yang approach.

\begin{proposition}
For fixed $A\in \II(\RR)$ and an interval continuous function $F$,
The function $G:\mathcal{I}_A\lra \II(\RR)$ defined by

$$G(Y)=\int_A^Y F(X) dX,$$

\noindent where $\mathcal{I}_A=\{Y\in\II(\RR)\;\mid\; A\leq Y\}$,
is continuous.
\end{proposition}

\proof{ It is enough to show that $\pi_1\circ G$ and $\pi_2\circ
G$ are continuous. Since $F$ is continuous we have that $F_l$ and
$F_r$ are continuous as well. Therefore, the real functions
$f_l(\underline{y})={\displaystyle
\int_{\underline{a}}^{\underline{y}}F_l(x)dx}$ and
$f_r(\underline{y})={\displaystyle
\int_{\underline{a}}^{\underline{y}}F_r(x)dx}$ are continuous. On
the other hand, the function $H:\mathcal{I}_A\lra \RR$ given by
\linebreak $H(Y)=\frac{\displaystyle d_M(A,Y)}{\displaystyle
\underline{y}-\underline{a}}=\max\left \{1, \frac{\displaystyle
\overline{y}-\overline{a}}{\displaystyle
\underline{y}-\underline{a}}\right \}$ is clearly continuous.
Thus, the functions $\pi_1\circ G=H\cdot (f_l\circ \pi_1)$ and
$\pi_2\circ G=H\cdot(f_r\circ \pi_1)$ are continuous as well.
\hfill \small $\blacksquare$}

\begin{proposition}
Let $F_2$ and $F_1$ be integrable functions from
$\mathcal{I}_{[A,B]}$ to $\II(\RR)$. If

$$F_2\left (\left [x,\overline{a}+\frac{\displaystyle
\overline{b}-\overline{a}}{\displaystyle
\underline{b}-\underline{a}}(x-\underline{a})\right ]\right
)=F_1\left (\left [x,\overline{a}+\frac{\displaystyle
\overline{b}-\overline{a}}{\displaystyle
\underline{b}-\underline{a}}(x-\underline{a})\right ]\right )$$

\noindent for all $x\in [\underline{a}, \underline{b}]$ then

$$\int_A^B F_2(X)dX=\int_A^B F_1(X)dX.$$
\end{proposition}

\proof{ Straightforward from Corollary ~\ref{coro-charac-gen} and
definition of $F_l$ and $F_r$ functions. \hfill \small
$\blacksquare$}

\subsection{Improper Interval Integrals}

 The definite integral is defined to integrate that continuous functions on a limited and closed interval. However, for those functions that has only one point of  discontinuity (actually, for a enumerate number of point of descontinuity) it is also possible to integrate it by means using the concept of improper integral. Such cases arise for the following types of intervals: $[a,+\infty)$, $(-\infty,b]$ or $(-\infty,+\infty)$. 

\begin{definition}
Let $F:\mathcal{I}_{[A,B]}\lra \II(\RR)$ be an interval function. An interval  $L \in  \II(\RR)$ is called to be the \textbf{limit} of $F$ when $X \in \mathcal{I}_{[A,B]}$ tends to the interval $C$  and it is denoted by $L = \displaystyle \lim_{X \to C}F(X)$, if for every given $\varepsilon > 0$ there exists an $\delta > 0$ such that $d_M(F(X),L) < \varepsilon$ whereas $0 < d_M(X,C) < \delta$ and $X \in \mathcal{I}_{[A,B]}$.
\end{definition}

Some researchers have been worked on the definition of limits of an interval function and its properties. Here we will not deal with this subject in depth, but for the most interested reader on this topic we recommend \cite{Wu98}.

Considering the intervals $[\infty]^{x} = [x,+\infty]$ and $[\infty]_{x} = [-\infty,x]$ for some $x \in \RR$. Particularly, when $x=+\infty$ interval $[\infty]^{+\infty} = [+\infty,+\infty]$ will be denoted just by $[+\infty]$. Similarly  $[-\infty]  = [-\infty,-\infty]$. Using that notation, we are able to define limit for infinite. So it follows that
$$
\mathcal{I}_{[A,[\infty]^{b}]} = \{X \in \II(\RR) \ | \ A \leq X \leq [b,+\infty] \ and \ \underline{a} \leq b  \}
$$
and
$$
\mathcal{I}_{[[\infty]_{a},B]} = \{X \in \II(\RR) \ | \ [-\infty,a] \leq X \leq B \ and \ a \leq \overline{b}  \}
$$

\begin{definition} \label{limit_infty}
Let $F:\mathcal{I}_{[A,B]}\lra \II(\RR)$ be an interval function and  $L \in  \II(\RR)$. Thus
\begin{enumerate}
\item If $B = [\infty]^{b}$ for some $b \in \RR$ then  $L = \displaystyle \lim_{X \to  [\infty]^{b}}F(X)$ if given $\varepsilon > 0$ there exists an $0 < K \in \mathcal{I}_{[A,B]}$ such that for every $X > K$ it follows that  $d_M(F(X),L) < \varepsilon$;
\item If  $A = [\infty]_{a}$ for some $a \in \RR$ then  $L = \displaystyle \lim_{X \to  [\infty]_{a}}F(X)$ if given $\varepsilon > 0$ there exists an $0 < K \in \mathcal{I}_{[A,B]}$ such that for every $X < K$ it follows that  $d_M(F(X),L) < \varepsilon$;
\end{enumerate}
\end{definition}

It is clear that in case $B = [+\infty]$ and $A = [-\infty]$ the limits on infinite $\displaystyle \lim_{X \to  [+\infty]}F(X)$ and $\displaystyle \lim_{X \to  [-\infty]}F(X)$ are particular cases of the Definition \ref{limit_infty}.

So now we have the main conditions to define the notion of improper interval integral as follows.

\begin{definition} \label{improper_int}
Let $F:\mathcal{I}_{[A,B]}\lra \II(\RR)$ be an interval function. Thus
\begin{enumerate}
	\item If $B = [\infty]^{b}$ for some $b \in \RR$ then $$\int_A^{[\infty]^{b}} F(X)dX = \lim_{C \to [\infty]^{b}} \int_A^{C} F(X)dX$$
	\item If $A = [\infty]_{a}$ for some $a \in \RR$ then $$\int_{[\infty]_{a}}^B F(X)dX = \lim_{C \to [\infty]_{a}} \int_C^{B} F(X)dX$$
	\item For $A = [-\infty]$ and  $B = [+\infty]$  then $$\int_{[-\infty]}^{[+\infty]} F(X)dX = \lim_{C \to [-\infty]}   \int_C^{[+\infty]} F(X)dX.$$
	\end{enumerate}
\end{definition}

It is worth to note that 
$$\int_{[-\infty]}^{[+\infty]} F(X)dX = \lim_{D \to [+\infty]}   \int_{[-\infty]}^D F(X)dX$$

\section{Interval Probability Density Functions}


Interval probability,  is an extension of classical probability by
consider interval-values to represent the probability of some
event. This area is not new, an in fact there are several approach
for interval probability, among them we have
\cite{Cam00,CB90,Dem67,Rus87,Tes92,Wei00}. The interval
probability of an event is natural in some situations:

\begin{enumerate}
\item  When is working with interval data as for example in
\cite{Cam13,Kov16}.

\item  When is used a Von Mises frequentist approach to
probability which is obtained via observations by projection of a
future stability, once  ``the relative frequency of the observed
attributes would tend to a fixed limit if the observations were
continued indefinitely by sampling from a collective''
\cite{SB08}. Nevertheless, since we can not observed indefinitely,
this projection include an error which could be captured by an
interval.

\item When the probability of an event is not representable
finitely, and therefore is necessary  use an approximation of this
probability \cite{Cam00}.

\item When we has an uncertainty in the probability value, for
example, when only partial information about error distributions
is available and standard statistical approaches cannot be applied
\cite{Kre04}. 

\end{enumerate}

An important notion in standard probability theory is the concept
of  probability density  function  of a continuous random
variable, which is a function that describes the relative
likelihood for the continuous random variable to occur at a given
point in the observation space.

\subsection{Interval Probability Spaces}

 Let $\mathcal{F}$ be a $\sigma$-algebra in the standard sense over a set $\Omega$ (the sample space).
 A (positive) \emph{interval measure} $\mu :
 \mathcal{F} \rightarrow\II(\RR)^+$, where $\II(\RR)^+=\{X \in \II(\RR)\;\mid\; [0,0]\leq X\}$,
  is a function that assigns a (positive) real interval to
each element of $\mathcal{F}$ %
such that the following properties are satisfied:

\begin{description}
\item [(i)] The empty set has measure zero:
$\mu(\emptyset)=[0,0]$

\item [(ii)] Countable additivity: if $\{A_i\;\mid\; i\in
I\}\subseteq \mathcal{F}$ is a set of pairwise disjoints for some
countable index set $I$, then
\begin{equation} \label{eq-part-sum}
\mu\left (\bigcup_{i\in I} A_i \right )=\sum_{i\in I} \mu(A_i)
\end{equation}
\end{description}

An \emph{interval measurable space} is defined as a tuple
$\prt{\Omega, \mathcal{F}, \mu}$, where $\Omega$ is a set,
$\mathcal{F}$ is a $\sigma$-algebra over $\Omega$ and $\mu$ is an
interval measure on $\prt{\Omega,\mathcal{F}}$. The elements of
$\mathcal{F}$ are called \emph{interval measurable sets}.


An \emph{interval probability space} $\prt{\Omega, \mathcal{F},
\mathcal{P}}$ is a tuple consisting of a sample space $\Omega$, a
$\sigma$-algebra $\mathcal{F}$ of subsets of $\Omega$, and a
positive,  interval measure $\mathcal{P}$ on
$\prt{\Omega,\mathcal{F}}$ satisfying $\mathcal{P} (\Omega) =
[1,1]$. In this case, $\Omega$ is known as the \emph{outcome
space} or \emph{sample space}, $\mathcal{F}$ is called the set of
\emph{events} and $\mathcal{P}$ is called an \emph{interval
probability measure} or by simplicity an interval probability.

\subsection{Interval Density Functions}

Let $\Omega$ be a sample space.random An interval random variable
$\mathcal{X}$ is a function $\mathcal{X}:\Omega\rightarrow
\II(\RR)$, i.e. assign a real interval to each sample point in
$\Omega$. When $\mathcal{X}$ is a continuous interval function
which also has a piecewise continuous derivative
$dF_\mathcal{X}(A)/d(A)$ we called $\mathcal{X}$ of continuous
interval random variable.

Let $\mathcal{X}$ be an interval random variable and $A$ and $B$
fixed intervals such that $A\leq B$. Then define the events:

\begin{eqnarray}
(\mathcal{X}=A)=\{\zeta \mid \mathcal{X}(\zeta)= A\} \\
(\mathcal{X}\leq A)=\{\zeta \mid \mathcal{X}(\zeta)\leq A\} \\
(\mathcal{X}\gg A)=\{\zeta \mid \mathcal{X}(\zeta)\gg A\} \\
(A\ll \mathcal{X}\leq B)=\{\zeta \mid A\ll \mathcal{X}(\zeta)= B\} \\
\end{eqnarray}

Given an interval probability space $\prt{\Omega, \mathcal{F},
\mathcal{P}}$, these events have associated the following interval
probabilities

\begin{eqnarray}
P(\mathcal{X}=A)=\mathcal{P}\{\zeta \mid \mathcal{X}(\zeta)= A\} \\
P(\mathcal{X}\leq A)=\mathcal{P}\{\zeta \mid \mathcal{X}(\zeta)\leq A\} \\
P(\mathcal{X}\gg A)=\mathcal{P}\{\zeta \mid \mathcal{X}(\zeta)\gg A\} \\
P(A\ll \mathcal{X}\leq B)=\mathcal{P}\{\zeta \mid A\ll \mathcal{X}(\zeta)= B\}
\end{eqnarray}

 The \emph{interval distribution function} of $\mathcal{X}$ is the function
 $F_\mathcal{X}:\II(\RR)\rightarrow \II(\RR)$ defined by

 \begin{equation}\label{eq-F_X}
 F_\mathcal{X}(X)=P(\mathcal{X}\leq X).
 \end{equation}

The following properties are straightforward from the fact that
$F_\mathcal{X}$ is based on an interval  probability space.

\begin{itemize}
\item [(F1)] $[0,0]\leq F_\mathcal{X}(X) \leq [1,1]$

\item [(F2)] If $X\leq Y$ then $F_\mathcal{X}(X)\leq F_\mathcal{Y}(B)$

\item [(F3)] $\Lim{X\rightarrow [+\infty]} F_\mathcal{X}(X)=[1,1]$

\item [(F4)] $\Lim{X\rightarrow [-\infty]} F_\mathcal{X}(X)=[0,0]$

\item [(F5)] $\Lim{X\rightarrow A^+} F_\mathcal{X}(X)=F_\mathcal{X}(A)$,
where $A^+=\Lim{[0,0]\ll \epsilon\rightarrow [0,0]}A+\epsilon$
\end{itemize}

From Eq. (\ref{eq-F_X}) is possible to obtain the probability of
other events:

\begin{itemize}
\item [(P1)] $P(A\ll \mathcal{X} \leq B)
=F_\mathcal{X}(B)-F_\mathcal{X}(A)$

\item [(P2)] $P(\mathcal{X}\gg A)=[1,1]-F_\mathcal{X}(A)$

\item [(P3)] $P(\mathcal{X}\ll B)=F_\mathcal{X}(B^-)$ where
$B^-=\Lim{[0,0]\ll \epsilon \rightarrow [0,0]} B-\epsilon$.
\end{itemize}

Let $\mathcal{X}$ be a continuos random variable. The function
$f_\mathcal{X}:\II(\RR)\rightarrow \II(\RR)$ defined by

\begin{equation} \label{eq-f_X}
f_\mathcal{X}(X)=\frac{dF_\mathcal{X}(X)}{dX}
\end{equation}
is called the \emph{interval probability density function} of
$\mathcal{X}$.

\begin{proposition}
The interval probability density function of $\mathcal{X}$ satisfy
the following properties:

\begin{enumerate}
\item $f_\mathcal{X}(X)\geq [0,0]$

\item $\displaystyle \int_{[-\infty]}^{[+\infty]}f_\mathcal{X}(X)dX=[1,1]$

\item $F_\mathcal{X}$ is piecewise continuous

\item $P(A\ll \mathcal{X}\leq B)=\displaystyle \int_A^B f_\mathcal{X}(X)dX$
\end{enumerate}

\end{proposition}

\begin{proof}
\begin{enumerate}
\item By considering Properties (F1) and (F2) and Eq. (\ref{eq-f_X}) it is clear that  $f_\mathcal{X}(X)\geq [0,0]$.
\item Taking into account Definition \ref{improper_int} and Properties (F3) and (F4)  it follows that

$$
\begin{array}{rcl}
\displaystyle \int_{[-\infty]}^{[+\infty]}f_\mathcal{X}(X)dX           & = & \displaystyle  \lim_{C \to [-\infty]} \int_C^{[+\infty]} f_\mathcal{X}(X)dX \\                                                                            & = & \displaystyle  \lim_{C \to [-\infty]} \lim_{D \to [+\infty]} \int_C^{D} f_\mathcal{X}(X)dX \\ 
& = & \displaystyle  \lim_{C \to [-\infty]} \lim_{D \to [+\infty]} \left |F_\mathcal{X}(D)- F_\mathcal{X}(C)\right |_C^D\\ 
& = & \left |\lim_{D \to [+\infty]} F_\mathcal{X}(D)- \lim_{C \to [-\infty]}  F_\mathcal{X}(C) \right|\\ 
& = & [1,1]-[0,0] = [1,1]
\end{array}
$$
\item Straightforward Definition of continuous interval random variable.
\item The distribution function $F_\mathcal{X}$ of a continuous interval
random variable $\mathcal{X}$ can be obtained by
\begin{equation} \label{eq-FX-int-fX}
F_\mathcal{X}(X)=P(\mathcal{X}\leq X)=\int_{[-\infty]}^X
f_\mathcal{X}(\xi)d\xi
\end{equation}
Thus, once when $\mathcal{X}$ is a continuous interval random
variable $P(\mathcal{X}=X)=[0,0]$, then
\begin{equation} \label{eq16}
P(A\ll \mathcal{X}\leq B)=P(A\leq \mathcal{X}\leq B)=P(A\leq
\mathcal{X}\ll B)=P(A\ll \mathcal{X}\ll B)=\int_A^B
f_\mathcal{X}(X)dX
\end{equation}
\end{enumerate}
\end{proof}

Also, from Eq. (\ref{eq16}) it possible to verify that
\begin{equation}
P(A\ll \mathcal{X}\leq B) = \int_A^B
f_\mathcal{X}(X)dX =  F_\mathcal{X}(B)-F_\mathcal{X}(A)
\end{equation}
The mean or expected value of a continuous interval random
variable $\mathcal{X}$, denoted by $\mu_\mathcal{X}$, is defined
by

\begin{equation}
\mu_\mathcal{X}=\int_{[-\infty]}^{[+\infty]}
Xf_\mathcal{X}(X)dX
\end{equation}

The $n$th moment of a continuous interval random variable
$\mathcal{X}$, denoted by $E(\mathcal{X}^n)$, is defined by

\begin{equation}
E(\mathcal{X}^n)=\int_{[-\infty]}^{[+\infty]}
X^nf_\mathcal{X}(X)dX
\end{equation}

Thus, the mean of $\mathcal{X}$ is the first moment of
$\mathcal{X}$.

The variance of a continuous interval random variable
$\mathcal{X}$, denoted by $Var(\mathcal{X})$, is defined by

\begin{equation}
Var(\mathcal{X})=\int_{[-\infty]}^{[+\infty]}
(X-\mu_\mathcal{X})^2f_\mathcal{X}(X)dX
\end{equation}

The standard deviation of a continuous interval random variable
$\mathcal{X}$, denoted by $\sigma_\mathcal{X}$, is defined by

\begin{equation}
\sigma_\mathcal{X}(X)=\sqrt{\int_{[-\infty]}^{[+\infty]}
(X-\mu_\mathcal{X})^2f_\mathcal{X}(X)dX}
\end{equation}

\subsection{Uniform, Exponencial and Gaussian Interval Probability Distribution Functions}

Let $\mathcal{X}$ be a continuous interval random variable and $A$
and $B$ intervals such that $A\ll B$. $\mathcal{X}$ is uniform
over $(A,B)$ if

\begin{equation}
 f_\mathcal{X}(X)=\left\{ \begin{array}{ll}
\frac{[1,1]}{B-A} & \mbox{ if }A\ll X\ll B \\

[0,0] & \mbox{ otherwise} \\
\end{array} \right.
\end{equation}

Notice that, in this case the interval distribution function of
$\mathcal{X}$, by consider the Eq. (\ref{eq-FX-int-fX}), is

$$\begin{array}{ll}
F_\mathcal{X}(X) & =\displaystyle \int_{[-\infty]}^X f_\mathcal{X}(\xi)d\xi \\
& = \left\{\begin{array}{ll}
          \displaystyle \int_A^X f_\mathcal{X}(\xi)d\xi & \mbox{ if }A\ll X\ll B \\

          [0,0] & \mbox{ if $\underline{X} \leq \underline{A}$ or }\overline{X}\leq \overline{a} \\

          [1,1] & \mbox{ if }B\leq X \\
\end{array}
\right . \end{array}$$ where

$$\begin{array}{ll}
 \displaystyle \int_A^X f_\mathcal{X}(\xi)d\xi & =
 \left[ \displaystyle \int_{\underline{a}}^{\underline{x}}
 f_{\mathcal{X}l}(x)dx,\int_{\underline{a}}^{\underline{x}}
 f_{\mathcal{X}r}(x)dx \right]  \frac{d_M(A,X)}{\underline{x}-\underline{a}} \\
 & = \left[ \displaystyle \int_{\underline{a}}^{\underline{x}}
 \pi_1(f_{\mathcal{X}}([x,\overline{a}+\frac{\overline{x}-\overline{a}}{\underline{x}-\underline{a}}(x-\underline{a})]))dx, \right. \\
 & \,\,\,\,\,\,\, \left. \displaystyle \int_{\underline{a}}^{\underline{x}}
 \pi_2(f_{\mathcal{X}}([x,\overline{a}+\frac{\overline{x}-\overline{a}}{\underline{x}-\underline{a}}(x-\underline{a})]))dx \right]
 \frac{d_M(A,X)}{\underline{x}-\underline{a}}\\
 & = \left[\displaystyle  \int_{\underline{a}}^{\underline{x}}
 \pi_1(\frac{[1,1]}{B-A})dx,\int_{\underline{a}}^{\underline{x}}
 \pi_2(\frac{[1,1]}{B-A})dx \right]
 \frac{d_M(A,X)}{\underline{x}-\underline{a}}\\
   & = \left [\displaystyle \int_{\underline{a}}^{\underline{x}}
\frac{1}{\overline{b}-\underline{a}} dx,
\int_{\underline{a}}^{\underline{x}}
\frac{1}{\underline{b}-\overline{a}}dx \right ]
\frac{d_M(A,X)}{\underline{x}-\underline{a}} \\
&= \left [
\frac{x-\underline{a}}{\overline{b}-\underline{a}},
\frac{x-\underline{a}}{\underline{b}-\overline{a}} \right ] \frac{d_M(A,X)}{\underline{x}-\underline{a}} \\
\end{array}$$

The uniform distribution is naturally used when has no previous
knowledge of the begin and end for where a variable could take
values but the probability that it take a value in the range is
the same. For example, random numbers in an interval which are
distributed according to the standard uniform distribution. But if
the limits not are well defined,  then in this case is more
adequate use standard uniform interval distribution.

Let $\mathcal{X}$ be a continuous interval random variable and
$\lambda \in \II(\RR)$ such that $[0,0]\ll \lambda$. $\mathcal{X}$
is exponential with parameter $\lambda$ if

\begin{equation}
 f_\mathcal{X}(X)=\left\{ \begin{array}{ll}
\lambda e^{-\lambda X} & \mbox{ if }[0,0]\ll X \\

[0,0] & \mbox{ otherwise} \\
\end{array} \right.
\end{equation}
where $e^X=\left [e^{\underline{X}},e^{\overline{X}} \right ]$. Thus,
$f_\mathcal{X}(X)=\left [\underline{\lambda} e^{-\overline{\lambda}
\overline{X}},\overline{\lambda} e^{-\underline{\lambda}
\underline{X}} \right ]$ when $X\gg [0,0]$

 Notice that, in this case the interval distribution function
of $\mathcal{X}$, by consider the Eq. (\ref{eq-FX-int-fX}), is

$$\begin{array}{ll}
F_\mathcal{X}(X) & = \displaystyle \int_{[-\infty]}^X f_\mathcal{X}(\xi)d\xi \\
& = \left\{\begin{array}{ll}
          \displaystyle  \int_{[0,0]}^X f_\mathcal{X}(\xi)d\xi & \mbox{ if }[0,0]\ll X\\

          [0,0] & \mbox{ otherwise} \\
\end{array}
\right . \end{array}$$ where

$$\begin{array}{ll}
\displaystyle  \int_{[0,0]}^X f_\mathcal{X}(\xi)d\xi & = \left[
\displaystyle  \int_{0}^{\underline{x}}
 f_{\mathcal{X}l}(x)dx,\int_{0}^{\underline{x}}
 f_{\mathcal{X}r}(x)dx \right]  \frac{d_M([0,0],X)}{\underline{x}} \\
 & = \left[\displaystyle \int_{0}^{\underline{x}}
 \pi_1(f_{\mathcal{X}}([x,\frac{\,\,\overline{x}\,\,}{\underline{x}}x]))dx,\int_{0}^{\underline{x}}
 \pi_2(f_{\mathcal{X}}([x,\frac{\,\,\overline{x}\,\,}{\underline{x}}x]))dx \right]
 \frac{\,\,\overline{x}\,\,}{\underline{x}}\\

 & = \left[\displaystyle \int_{0}^{\underline{x}}
 \underline{\lambda} e^{-\overline{\lambda}
\left(\frac{\,\,\overline{x}\,\,}{\underline{x}}x\right)}
dx,\int_{0}^{\underline{x}} \overline{\lambda}
e^{-\underline{\lambda} x} dx \right]\frac{\,\,\overline{x}\,\,}{\underline{x}} \\
& =  \displaystyle \left[-\frac{\,\,\underline{\lambda}\underline{x}\,\,}{\overline{\lambda}\overline{x}}(e^{-\overline{\lambda}\overline{x}} - 1), -\frac{\,\,\overline{\lambda}\,\,}{\underline{\lambda}}(e^{-\underline{\lambda}\underline{x}} - 1) \right] \frac{\,\,\overline{x}\,\,}{\underline{x}}\\
& =  \displaystyle \left[\frac{\,\,\underline{\lambda}\,\,}{\overline{\lambda}}(1 - e^{-\overline{\lambda}\overline{x}}), \frac{\,\,\overline{\lambda}\overline{x}\,\,}{\underline{\lambda} \underline{x}}(1 - e^{-\underline{\lambda}\underline{x}}) \right] \\
\end{array}$$

Let $\mathcal{X}$ be a continuous interval random variable.
$\mathcal{X}$ is normal or Gaussian if

\begin{equation}
 f_\mathcal{X}(X)=\frac{1}{\sqrt{2\pi
 Var(\mathcal{X})}}e^{-\frac{(X-\mu_\mathcal{X})^2}{2
 Var(\mathcal{X})}}
\end{equation}

Notice that, in this case

\begin{equation}
 f_\mathcal{X}(X)=\left [\frac{1}{\sqrt{2\pi
 \overline{Var(\mathcal{X})}}}e^{-\frac{\,\,(\overline{X}-\underline{\mu_\mathcal{X}})^2\,\,}{2
 \underline{Var(\mathcal{X})}}},\frac{1}{\sqrt{2\pi
 \underline{Var(\mathcal{X})}}}e^{-\frac{\,\,(\underline{X}-\overline{\mu_\mathcal{X}})^2\,\,}{2
 \overline{Var(\mathcal{X})}}}\right ]
\end{equation}

 Notice that, in this case the interval distribution function
of $\mathcal{X}$, by consider the Eq. (\ref{eq-FX-int-fX}), is

$$\begin{array}{ll}
F_\mathcal{X}(X) & = \displaystyle \int_{[-\infty]}^X f_\mathcal{X}(\xi)d\xi \\
& = \left[ \displaystyle  \int_{-\infty}^{\underline{x}}
 f_{\mathcal{X}l}(x)dx,\int_{-\infty}^{\underline{x}}
 f_{\mathcal{X}r}(x)dx \right]  \frac{d_M(A,B)}{\underline{b}-\underline{a}} \\
 & = \left[\displaystyle  \int_{-\infty}^{\underline{x}}
 \pi_1(f_{\mathcal{X}}([x,\overline{a}+\frac{\overline{b}-\overline{a}}{\underline{b}-\underline{a}}(x-\underline{a})]))dx, \right. \\
 & \,\,\,\,\,\, \left. \displaystyle \int_{-\infty}^{\underline{x}}
 \pi_2(f_{\mathcal{X}}([x,\overline{a}+\frac{\overline{b}-\overline{a}}{\underline{b}-\underline{a}}(x-\underline{a})]))dx \right]
 \frac{d_M(A,B)}{\underline{b}-\underline{a}}\\
\end{array}$$

\section{Final Remarks}

As can be seen in an evaluation of the existing literature regarding probability density functions, the interval mathematics was still little explored in 
this scope, although it is an important way of estimating the imprecision of the parameters. The few published works on this topics (as one can see in 
\cite{DB93,PR05}) consider only the interval interpretation for measuring the uncertainty of the variables taking into account a classical view of 
probability density functions instead of an interval one. For this reason, the concept of interval probability density functions based on a new way of 
defining interval integrals presented in this paper creates a wide range of possibilities for probabilistic problems from the interval point of view. The 
results show that the theory is consistent and allows us to generate an interesting way to control the inaccuracies and uncertainties of the variables throughout the mathematical process of the model.

Under point of view of the interval integral theory presented here it is important to highlight that there exists functions which are continuous (according to the Moore topology) but are not inclusion monotonic, for example the interval function $F(X)=m(X)+\frac{1}{2}(X-m(X))$, where $m(X)$ is the middle point of $X$\footnote{This function was used by Moore in \cite{Moo79} as an example of an interval valued function which is not inclusion monotonic, lately in \cite{SBA06} was proved that this function is continuous with respect to the Moore topology.}. However the (real) integrals based on  Riemann sums are defined for all continuous functions, and therefore would be desirable that an interval extension of this notion consider all continuous interval functions (with respect to a suitable notion of continuity for interval functions) and not only those which are inclusion monotonic. So, the deletion of inclusion monotonic restriction, in Moore and Yang integral definition, made in our extension is fundamental in order to provide a robust extension and to consider the integral of this kind of interval functions.

Another restriction in the Moore and Yang approach,  which also occurs  in Caprani, Madsen and Rall integrals in \cite{CMR81}, is by considering only real numbers as integral limits. The Corliss
extension in \cite{Cor87} of Caprani, Madsen and Rall integral in \cite {CMR81} for interval integral limits, consist in a couple of Caprani, Madsen and 
Rall integrals with real integration limits.
Therefore, it could be used in order to extend any other notion of interval integral where the limits are real numbers by simply substituting the Caprani, Madsen and Rall integral for another notion. Nevertheless, these extension and any other of the same line, still computationally more easy to calculate, is not intrinsic and lack of mathematical foundations. Our extension
follows the original spirit of Moore and Yang, i.e. it is a generalization of the usual Riemann sum integrals based on the extension of all elements used in such kind of integral and
therefore is mathematically a well founded extension.

The
 Corollary \ref{coro-charac-gen} can be seen as a meta-algorithm
to compute our extension of the Moore-Yang integrals, in the sense
that we can use any usual method to compute the Riemann integrals
for each integral in the extremes of the interval in equation
(\ref{eq-charac-gen}). Obviously, in this case, as is usual in
interval computing, we need to use directed rounding in each path
of the computing.

Choquet integrals and generalizations of the Choquet integrals \cite{Bus21,Dimuro20,Dimuro20b} are an important family of (pre)aggregation functions used to merge discrete real inputs. Nevertheless,  in many real problems, the inputs arise in the continuum and therefore the standard  Choquet integrals and their generalizations are not adequate. In \cite{Jin18}  was proposed a way of merging Riemann integrable inputs from discrete Choquet integral. As future work we will extend such work by considering interval-valued extensions of  Choquet integrals \cite{Bus13,Gong20,Pater19}.


\bibliographystyle{plain}

\begin{thebibliography}{99}


\bibitem{AB97} B.~M. Aci\'oly,  B.~C. Bedregal.   \newblock {A
quasi-metric topology compatible with inclusion monotonicity on
interval space}. \newblock {\em Reliable Computing}, 3:305--313,
1997.

\bibitem{Asmus17}T. da C. ~Asmus, G. P. ~Dimuro, B. C. ~Bedregal.
On Two-Player Interval-Valued Fuzzy Bayesian Games. \emph{Int. J. Intell. Syst.} 32(6): 557-596, 2017.


\bibitem{Bed13} B.~C. Bedregal, R.~H.~N. Santiago. Some continuity notions for interval functions and representation. 
\emph{Computational and Applied Mathematics}, 32(3):435-446, 2013. 

\bibitem{DB93}  D. ~Berleant. \newblock {Automatically Verified Reasoning with Both Intervals and Probability Density Functions}.
\newblock {\em Interval Computations}, 2:48 -- 70, 1993.

\bibitem{Bus13} H. Bustince, M. Galar, B.~C. Bedregal, A. Koles\'arov\'a, R. Mesiar.
A new approach to interval-valued Choquet integrals and the problem of ordering in interval-valued fuzzy set applications. 
\emph{IEEE Trans. on Fuzzy Systems}, 21(6):1150--1162, 2013.


\bibitem{Bus21} H. Bustince, R. Mesiar, J. Fernandez, M. Galar, D. Paternain, A. H. Altalhi, G. P. Dimuro, B. Bedregal and  Z. Tak\'ac.
d-Choquet integrals: Choquet integrals based on dissimilarities. {\em Fuzzy Sets and Systems}, 414: 1--27, 2021.



\bibitem{Cam00} M.~A. Campos. \newblock {The Interval Probability: Applications to discrete random variables}.
\newblock {\em Tend\^encias em Matem\'atica Aplicada e Computacional},
1(2):333--344, 2000.

\bibitem{Cam13} M.~A. Campos, M.G. dos Santos. Interval probabilities and enclosures. \emph{Computational and Applied Mathematics},  32:413--423, 2013. 

\bibitem{CMR81} O. Caprani, K. Madsen, L.~B.
Rall. \newblock {Integration of interval function}. \newblock {\em
SIAM Journal on Mathematical Analysis}, 12:321--341, 1981.

\bibitem{Cha15} Y. Chalco-Cano, W.~A. Lodwick, W. Condori-Equice.
Ostrowski type inequalities and applications in numerical integration for interval-valued functions. \emph{Soft Computing,} 19(11): 3293--3300, 2015.

\bibitem{Cor87} G.~F. Corliss.  \newblock {Computing Narrow Inclusions for
definite Integrals}. \newblock In E.~W. Kaucher, U. Kulisch, and
C. Ullrich, editors, {\em  Computer Arithmetic: Scientific
Computation and Programming Languages},   Teubner, Stuttgart,
pp.150--169, 1987.




\bibitem{CB90}  W. Cui, D.~I. Blockley. \newblock {Interval probability theory for evidential support}.
\newblock {\em  International Journal of Intelligent Systems}, 5(2):183 -- 192,
1990.


\bibitem{Costa13}
 C.G.  da Costa, B. C. Bedregal and A. D. D\'oria Neto. Atanassov's intuitionistic fuzzy probability and Markov chains, \emph{Knowl.-Based Syst.}, 43:52--62, 2013.

\bibitem{Dem67} A.~P. Dempster.  \newblock {Upper and lower probabilities induced by a multivalued
mapping}. \newblock {\em  The Annals of Mathematical Statistics},
38(2):325--339, 1967.



\bibitem{Dimuro20} 	G. P. Dimuro, J. Fernandez, B. Bedregal, R. Mesiar, J. A.  Sanz, G. Lucca and H. Bustince.
The state-of-art of the generalizations of the Choquet integral: From aggregation and pre-aggregation to ordered directionally monotone functions. {\em Information  Fusion}, 57: 27--43, 2020.

\bibitem{Dimuro20b} GP Dimuro, G Lucca, B Bedregal, R Mesiar, JA Sanz, CT Lin, H Bustince. Generalized CF1F2-integrals: from Choquet-like aggregation to ordered directionally monotone functions. {\em Fuzzy Sets and Systems}, 378:44--67, 2020.

\bibitem{Eda95} A. Edalat. \newblock {Domain theory and
integration}. \newblock {\em Theoretical Computer Science}, 151:
163--193, 1995.


\bibitem{Esc97} M.~H. Escard\'o.  \newblock {PCF extended with real
numbers: A domain  theoretic approach to higher-order exact real
number computation}. PhD thesis, University of London, Imperial
College of Science, Technology and Medicine, London, 1997.

\bibitem{Eva00} M. Evans, N. Hastings, B. Peacock. Statistical Distributions, 3rd ed. New York: Wiley, pp. 9-11, 2000. 


\bibitem{Fat19} F. Fatimah, D. Rosadi, R. B. F, Hakim and  J. C. R. Alcantud.
Probabilistic soft sets and dual probabilistic soft sets in decision-making. \emph{Neural Comput. Appl.} 31(S-1): 397--407, 2019.

\bibitem{Gong20} Z. Gong, X. Kou and T. Xie. Interval-valued Choquet integral for set-valued mappings: definitions, integral representations and primitive characteristics. {\em AIMS Mathematics},5(6): 6277--6297, 2020.

\bibitem{Intan07} R.~Intan. 
Interval probability of data querying based on fuzzy conditional probability relation. 
\emph{International Journal of Computer Science}, 34(1): 2007.

\bibitem{Jab16} A.~I. Jabbarova. 
Solution for the investment decision making problem through interval probabilities. 
\emph{Procedia Computer Science}, 102: 465--468, 2016.

\bibitem{JL20} K. D. Jamison, W. A. Lodwick.
A new approach to interval-valued probability measures, a formal method for consolidating the languages of information deficiency: Foundations. \emph{Information Sciences}, 507: 86--107 2020.

\bibitem{Jin18} L. Jin, M. Kalina, R. Mesiar and S. Borkotokey. Discrete Choquet Integrals for Riemann Integrable Inputs With Some Applications. \emph{IEEE Transactions on Fuzzy Systems}, 26(5): 3164--3169,  2018)

\bibitem{Kov16} B. Kovalerchuk, V. Kreinovich.
Comparison of formulations of applied tasks with intervals, fuzzy sets and probability approaches. In Proceedings of FUZZ-IEEE 2016, pp. 1478-1483.

\bibitem{Kre04} V. Kreinovich. \newblock {Probabilities, intervals, what next?
optimization problems related to extension of interval
computations to situations with partial information about
probabilities}. \newblock {\em Journal of Global Optimization} 29:
265--280, 2004.

\bibitem{KM81}  U.~W. Kulisch, W.~L.
Miranker. \newblock {\em Computer Arithmetic in Theory and
Practice}. \newblock Academic Press, 1981.

\bibitem{Li12} Y. Li, Y. Qin and Q. Lei. 
Confidence intervals for probability density functions under associated samples. 
\emph{Journal of Statistical Planning and Inference}, 142(6): 1516--1524, 2012.

\bibitem{Lod08} W.~A. Lodwick, K.~D. Jamison. 
Interval-valued probability in the analysis of problems containing a mixture of possibilistic, probabilistic, and interval uncertainty. 
\emph{Fuzzy Sets and Systems}, 159(21): 2845--2858, 2008.


\bibitem{Moo62}  R.~E. Moore. \newblock {Interval Arithmetic and Automatic
Error Analysis in Digital Computing}. \newblock  PhD thesis,
Stanford University, Stanford-CA, 1962.

\bibitem{Moo79} R.~E. Moore. \newblock {\em Methods and Applications for
Interval Analysis}. \newblock SIAM Studies in Applied Mathematics,
Philadelphia, 1979.


\bibitem{MSY60}  R.~E. Moore,  W. Strother, C.~T.
Yang. \newblock Interval Integrals. \newblock Technical Report
LMSD-703073, Lockheed Aircraft Corporation, Missiles and Space
Division, Sunnyvale-CA, 1960.

\bibitem{MY59} R.~E. Moore, C.~T. Yang. \newblock  Interval
Analysis I.  \newblock Technical Report LMSD-285, Lockheed
Aircraft Corporation, Missiles and Space Division, Sunnyvale-CA,
1959.


\bibitem{Pater19} D. Paternain, L. De Miguel, G. Ochoa, I. Lizasoain, R. Mesiar and H. Bustince.
The Interval-Valued Choquet Integral Based on Admissible Permutations. {\em IEEE Trans. on Fuzzy Systems}, 27(8): 1638--1647, 2019.


\bibitem{Ral82}  L.~B. Rall. \newblock {Integration of interval
functions II: The finite case}. \newblock {\em SIAM Journal on
Mathematical Analysis}, 13:690--697, 1982.

\bibitem{PR05}  P. ~Ram\'irez, J.~A. Carta. \newblock {Influence of the Data Sampling Interval in the Estimation of the Parameters of the Weibull Wind Speed Probability Density Distribution: a case study}.
\newblock {\em  Energy Conversion and Management}, 46:2419 -- 2438, 2005.

\bibitem{Rus87} E.H. Ruspini, \newblock {Approximate inference and interval
probabilities}. In Proceedings of the Conference on Information
Processing and Management of Uncertainty in Knowledge-Based
Systems, (Bonchon, B.; and Yager, R.R. Eds), 85--94, 1987.

\bibitem{SBA06}  R.~H.~N. Santiago,  B.~C. Bedregal,  B.~M. Aci\'oly.
\newblock {Formal Aspects of Correctness and Optimality of
Interval Computations}. \newblock  {\em Formal Aspects of
Computing}, 18:231--243, 2006.

\bibitem{Sarv98} V. Sarveswaran, J.~W Smith and D.~I. Blockey. \newblock {Reliability of corrosion-damaged steel structures using interval probability 
theory}. \newblock {\em Structural Safety}, 20(3), 237--255, 1998.

\bibitem{Sco70}  D.~S. Scott.  \newblock {Outline of a mathematical
theory of computation}. \newblock In {\em Proceeding of $4^{th}$
Annual Princeton Conference on Information Sciences and Systems},
pp. 169--176, 1970.

\bibitem{SB08} M. Schield, T.~V.~V. Burnham.  \newblock {Von Mises' frequentist approach to
probability}. \newblock In \emph{Proc. of 2008 Joint Statistical
Meetings}, Denver, 2008. Available at
http://www.statlit.org/pdf/2008SchieldBurnhamASA.pdf.


\bibitem{Sim63}  G.~F. Simmons, \newblock {\em Introduction to Topology
and Modern Analysis}. \newblock McGrall Hill, New York, 1963.

\bibitem{Tan04} H. Tanaka, K. Sugihara and Y. Maeda.  \newblock {Non-additive measures by interval probability functions}. \newblock {\em Information Sciences}
164(1-4):209--227, 2004.

\bibitem{Tes92}  B. Tessem. \newblock {Interval probability propagation}.
\newblock {\em  International Journal on Approximate Reasoning}, 7:95--120, 1992.



\bibitem{Wei00} K. Weichselberger. \newblock {The theory of interval-probability as a unifying concept for
uncertainty}. \newblock {\em Approximate Reasoning}, 24: 149--170,
2000.

\bibitem{Wu98} H. Wu. \newblock {The improper fuzzy Riemann integral and its numerical integration}. \newblock {\em Information Sciences}, 111(1-4): 109--137,
1998.

\bibitem{Yager84} R.~R. Yager. \newblock {Probabilities from fuzzy observations}. \newblock {\em Information Sciences}, 32: 1--32,
1984.

\bibitem{Yuan17} X. Yuan, J. Li and  X. Zhao.
Typical interval-valued hesitant fuzzy probability. 13th International Conference on Natural Computation, Fuzzy Systems and Knowledge Discovery (ICNC-FSKD) 2017, pp. 1182--1187


\bibitem{Zhang09} Q. Zang, B. Jia and S. Jiang. \newblock {Interval-valued intuitionistic fuzzy probabilistic set and some of its important properties}. \newblock {\em The 1st International Conference on Information Science and Engineering}, 2009.



\end{thebibliography}

\end{document}